\documentclass[10pt]{article}

\usepackage{eurosym}
\usepackage{hyperref}
\usepackage{amssymb}
\usepackage{amsfonts}
\usepackage{graphicx}
\usepackage{amsmath}
\usepackage{float,color,ulem}
\usepackage{epsf,epsfig,subfigure}
\usepackage{float,color,ulem}
\usepackage{ulem}
\usepackage{bbm}

\newfloat{figure}{H}{lof}
\newfloat{table}{H}{lot}
\floatname{figure}{\figurename}
\floatname{table}{\tablename}
\newtheorem{theorem}{Theorem}[section]

\newtheorem{assumption}[theorem]{Assumption}

\newtheorem{lemma}[theorem]{Lemma}

\newtheorem{proposition}[theorem]{Proposition}
\newtheorem{remark}[theorem]{Remark}

\hypersetup{colorlinks=true, linkcolor=blue, citecolor=red}
\newenvironment{proof}[1][Proof]{\textbf{#1.} }{\ \rule{0.5em}{0.5em}}
\newcommand{\R}{\mathbb{R}}
\numberwithin{equation}{section}
\allowdisplaybreaks
\def\R{\mathbb{R}}
\newcommand{\Li}{{\mathrm {L} } }
\def\dd{\mathrm{d}}
\newcommand{\xk}[1]{\left( #1 \right)}

\begin{document}
		\title{\textbf{Global asymptotic stability for  Gurtin-MacCamy's population dynamics model}}
	\author{\textsc{Zhaohai Ma$^{(a)}$ and Pierre Magal$^{(b)}$\thanks{Corresponding author. e-mail: \href{mailto:pierre.magal@u-bordeaux.fr}{pierre.magal@u-bordeaux.fr}}  }\\
			{\small \textit{$^{(a)}$School of Science, China University of Geosciences}},\\  
			{\small \textit{ Beijing 100083, People's Republic of China}} \\
		{\small \textit{$^{(b)}$Univ. Bordeaux, IMB, UMR 5251, F-33400 Talence, France.}} \\
		{\small \textit{CNRS, IMB, UMR 5251, F-33400 Talence, France.}}
	}
	\maketitle

%% \title[short text for running head]{full title}
%\title[Global Stability of an age-structured model]{\textbf{Global asymptotic stability for  Gurtin-MacCamy's population dynamics model}}
%
%
%
%%    author one information
%% \author[short version for running head]{name for top of paper}
%\author[Z. Ma]{Zhaohai Ma}
%\address{School of Science, China University of Geosciences,   Beijing 100083, People's Republic of China}
%%\curraddr{}
%\email{zhaohaima@cugb.edu.cn or zhaohaima@mail.bnu.edu.cn}
%%\thanks{}
%
%%    author two information
%\author[P. Magal]{Pierre Magal}
%\address{Univ. Bordeaux, IMB, UMR 5251, F-33400 Talence, France, CNRS, IMB, UMR 5251, F-33400 Talence, France.}
%\email{pierre.magal@u-bordeaux.fr}
%%\thanks{}

%    \subjclass is required.
 %\subjclass[2020]{Primary 34K20; 37C10; 37N25; 92D25 }
% \subjclass[2020]{{34K20; 37C10; 37N25; 92D25 } }

\date{}

%\dedicatory{}

%    "Communicated by" -- provide editor's name; required.
%\commby{}

%    Abstract is required.
\begin{abstract}
In this paper, we investigate the global asymptotic stability of an age-structured population dynamics model with a Ricker's type of birth function. This model is a hyperbolic partial differential equation with a nonlinear and nonlocal boundary condition. We prove a uniform persistence result for the semi-flow generated by this model. We obtain the existence of global attractors and we prove the global asymptotic stability of the positive equilibrium by using a suitable Lyapunov functional. Furthermore, we prove that our global asymptotic stability result is sharp, in the sense that Hopf bifurcation may occur as close as we want from the region global stability in the space of parameter.
\end{abstract}

\maketitle


 %\bibliographystyle{amsplain}

%\noindent \textbf{Keywords:} Age structure; population dynamics; Lyapunov function; Global stability; Uniform persistence; Global attractors.}

\section{Introduction}
The age-structured population dynamics models have a long history starting from the articles of Lotka \cite{Lotka} in (1907) and the article of Sharpe and Lotka \cite{Sharpe-Lotka} (1911).  The first nonlinear age-structured model was introduced by Gurtin and  MacCamy  \cite{Gurtin-Maccamy-1974-ARMA} in the 70th. They proved a result of  {the} existence and uniqueness of the solutions and a local stability result for an {equilibrium}. More recently, age-structured models received a lot of interest in the literature and we refer to the books \cite{Iannelli, Magal-Ruan2018, Thieme, Webb} for more results about this topic.

The article is devoted to the global asymptotic stability of the following Gurtin-MacCamy's age-structured model,
\begin{equation}\label{1.1}
   \left\{
     \begin{array}{ll}
       (\partial_t + \partial_a  ) u(t, a ) = - \mu u(t, a), & t> 0, a> 0,    \\
          u(t, 0) = \alpha f\left(\displaystyle \int_{0}^{\infty} \beta (a) u(t, a) \dd a \right), & t> 0,
     \end{array}
   \right.
\end{equation}
with the initial distribution
$
u(0,a)=u_0(a) \in \Li^{1}_+ \left( \left( 0, +\infty \right), \R \right).
$
Here the parameters $\alpha> 0, \mu > 0$ and $\beta$ and $f$ satisfy the following Assumption \ref{ASS1.1}.

\begin{assumption} \label{ASS1.1} We assume that
$
f(u)=ue^{-u}, u\geq 0
$
and $\beta \in \Li^{\infty}_+ (0, +\infty)$ is normalized by
\begin{equation*}
   \int_{0}^{+\infty} \beta(a) e^{-\mu a} \dd a = 1.
\end{equation*}
\end{assumption}
To explain our result, we first need to recall that the solution integrated along the characteristic of the system \eqref{1.1} is
\begin{equation*}%\label{2.1}
u(t, a) =
\left\{
\begin{array}{cc}
e^{-\mu t} u_0(a - t),  & \text{ if }  a \geq t,\\
e^{-\mu a} b(t -a), &  \text{ if }   a \leq  t,
\end{array}
\right.
\end{equation*}
where $t \to b(t)$ is the unique continuous solution of the Volterra's integral equation
\begin{equation}\label{1.2}
b(t) = \alpha f\left( \int_{t}^{\infty } \beta(a) e^{-\mu t} u_0(a - t) \dd a+ \int_{0}^{t } \beta(a) e^{-\mu a} b(t-a) \dd a\right), \forall t \geq 0.
\end{equation}
For each $\alpha > 1$, the system \eqref{1.1} has a unique positive equilibrium
\begin{equation} \label{1.3}
\overline{u}(a) = \overline{u}(0) e^{-\mu a },
\end{equation}
and  $\bar{u}(0) = \ln \alpha > 0$ is the unique strictly positive solution of the scalar equation
\begin{equation*}
\overline{u} (0) = \alpha f( \overline{u}  (0)) .
\end{equation*}
If we assume that  $\beta(.)\, e^{- \mu .} \overset{\ast}{\rightharpoonup}  \delta_{\tau}(a)$ for the weak star topology (in the dual of the space of some suitable exponential bounded continuous functions) and where $ \delta_{\tau}(a)$ is the Dirac mass at $\tau>0$. This  implies in particular
$$
\int_0^\infty \beta(a)\, e^{- \mu a} \varphi(a) \dd a \to \varphi(\tau), \forall \varphi \in C_c (0,+\infty),
$$
where $C_c (0,+\infty)$ is the space of continuous functions with compact support on $ (0,+\infty)$.

Then, by replacing $\beta(a)\, e^{- \mu a} $ by $\delta_{\tau}(a)$) in equation \eqref{1.2}, we obtain a continuous-time difference equation of the form
\begin{eqnarray} \label{1.4}
b(t)=\alpha f(b(t-\tau)), \forall t \geq \tau.
\end{eqnarray}

In Theorem \ref{TH-GAS} we will see that the positive equilibrium $\overline{u}$ is global stable for the system \eqref{1.1} whenever
\begin{equation} \label{1.5}
1 < \alpha \leq e^{2}.	
\end{equation}
As far as we know, global asymptotic stability and Lyapunov functional for the age-structured model was first considered by Magal, McCluskey, and Webb \cite{Magal-McCluskey-Webb}. We also refer to Smith and Thieme \cite{Smith-Thieme} for another approach for the same infection age-structured model. Since then, the Lyapunov method of \cite{Magal-McCluskey-Webb} has applied to many  {examples} of age-structured models in the literature. In the present article, we apply the recent results of Magal, Seydi, and Wang \cite{Magal-Seydi-Wang}, and we reconsider the decomposition of $L^1_+$ into two invariant subregions which  {are} used later on in the uniform persistence section. The invariance of subregions is easy to explain.
Indeed, our global stability result should reflect the fact that if we start only with old individuals, who can no longer reproduce, the population goes extinct. Therefore the state-space decomposition of $L^1_+$ should reflect this simple fact (see Section \ref{section3} for precisions).  This idea was neglected in many articles and we hope that this new method will help to understand this part of the problem.

One may also observe that condition \eqref{1.5} corresponds to the global asymptotic stability of the positive equilibrium $\overline{b}=\ln(\alpha)$ for the difference equation \eqref{1.4}. Moreover, $\alpha=e^2$ corresponds to  {a} Hopf's bifurcation point for the difference equation  \eqref{1.4} (when it is regarded as a  {discrete-time} problem). Indeed,
$$
\alpha f'(\overline{b})=\alpha [1- \ln(\alpha )] e^{- \ln(\alpha )}=-1.
$$
Hopf's bifurcation for  {the} model \eqref{1.1} was first considered by Magal and Ruan \cite{Magal-Ruan2009, Magal-Ruan2018} and Liu, Magal and Ruan \cite{Liu}. In Section \ref{section7}, we prove that for each $\alpha^\star > e^2$ (as close as we want from $e^2$) we can find a function $\beta(a)$ such that Hopf bifurcation occurs at $\alpha=\alpha^\divideontimes \in (e^2, \alpha^\star )$ for the system  \eqref{1.1}. In this sense, our global stability result is sharp with respect to the parameter $\alpha$.

The paper is organized as follows. In Section \ref{section2}, we state a dissipativity result. The invariant subregion under the semiflow generated by \eqref{1.1} is studied in Section \ref{section3}. In Section \ref{section4}, we consider the asymptotic smoothness of the semiflow generated by \eqref{1.1}.
In Section \ref{section5}, we prove that the semiflow generated by \eqref{1.1}  is uniformly persistent in the invariant subregion  $\widehat{M}_0$ and obtain the existence of the global attractor. In section \ref{section6}, we prove that the positive equilibrium is  global asymptotically stable by using the Lyapunov functional. Section \ref{section7} is devoted  {to} the existence of Hopf's bifurcation whenever $\alpha> e^2$.

\section{ Dissipativity}
\label{section2}
Set
$
\Lambda= \displaystyle \frac{\alpha}{\mu} \sup\limits_{x \geq 0}  f(x).
$
Consider the total number of individuals at time $t$
$$
U(t)=\int_0^\infty u(t,a)\dd a, \forall t\geq 0 \text{ and } U_0=\int_0^\infty u_0(a)\dd a.
$$
Then we have the following result about the total number of individuals $U(t)$ .

 \begin{lemma} \label{LE2.1} We have
$$
U(t)\leq e^{-\mu t}U_0+ \left(1- e^{-\mu t} \right) \Lambda, \forall t \geq 0.
$$
\end{lemma}

\section{Invariant subregion}
\label{section3}
In this section, we use  the recent approach presented in Magal, Seydi and Wang \cite{Magal-Seydi-Wang}. The following comparison result can be obtained by comparison of the Volterra integral equations.
\begin{theorem}\label{TH3.1}
Assume that $u_{\pm}(t,a)$ satisfies
\begin{equation} \label{3.1}
\left\lbrace
\begin{array}{l}
\partial_t u_{\pm}(t,a)+\partial_a u_{\pm}(t,a)=-\mu u_{\pm}(t,a), \text{ for } a \geq 0,\ t\geq0,\\
u_{\pm}(t,0)=\alpha \delta_\pm \int_0^{+\infty} \beta(a)u_{\pm}(t,a)\dd a, \\
u_{\pm}(0,.)=u_0 \in \Li^1_+(\left(0,+\infty \right), \mathbb{R}).
\end{array}
\right.
\end{equation}
Then
\begin{equation*} %\label{2.4}
u_{-}(t,a)\leq u(t,a) \leq u_{+}(t,a),\ \mbox{for}\ t\geq0, \ \mbox{for a. e.}\ a\geq 0,
\end{equation*}
where
$
\delta_-=e^{-\Vert \beta \Vert_{\Li^\infty} \max(U_0, \Lambda)} <1 = \delta_+.
$
\end{theorem}
\begin{proof}
We just need to observe that
$$
b(t)=\alpha \int_0^{+\infty} \beta(a)u(t,a)\dd a \exp\left( -\int_0^{+\infty} \beta(a)u(t,a)\dd a\right) {.}
$$
Therefore by using Lemma \ref{LE2.1}, we have
\begin{equation*}%\label{2.5}
\alpha \int_0^{+\infty} \beta(a)u(t,a)\dd a e^{-\Vert \beta \Vert_{\Li^\infty} \max(U_0, \Lambda)}  \leq b(t) \leq \alpha \int_0^{+\infty} \beta(a)u(t,a)\dd a.
\end{equation*}
%therefore by using the comparison result in \cite{Magal-Seydi-Wang} the result follows.
 Then the result follows by using the comparison result in \cite{Magal-Seydi-Wang}.
\end{proof}

Let
$$
\Gamma^{\pm}(a)=\alpha \delta_\pm \int_{a}^{\infty}e^{-\int_a^\theta[\mu+\lambda^{\pm}]\dd \sigma}\beta(\theta)\dd \theta,
$$ where $\lambda^{\pm}$ is chosen to satisfy that $\Gamma^{\pm}(0)=1$. That is to say that $\lambda^{\pm} \in \mathbb{R}$ satisfies
$$
\alpha  \delta_\pm \int_{0}^{\infty}e^{-\int_0^\theta[\mu+\lambda^{\pm}]\dd \sigma}\beta(\theta)\dd \theta=1.
$$
\begin{remark}  \label{REM3.2}
$\lambda^{\pm} \in \mathbb{R}$ is the dominant eigenvalue of  the linear operator associated to the system \eqref{3.1}. The above integral equation corresponds to the characteristic equation of this linear operator.
\end{remark}
The following lemma is showing that $\Gamma^{\pm}(a)$ is an adjoint eigenfunction associated the eigenvalues $\lambda^{\pm}$.
\begin{lemma} \label{LE3.3}
For each $t \geq 0$,
\begin{equation*}% \label{2.6}
\int_0^{\infty}\Gamma^{\pm}(a)u_{\pm}(t,a) \dd a=e^{\lambda^{\pm}t}\int_0^{\infty}\Gamma^{\pm}(a)u_0(a) \dd a.
\end{equation*}
\end{lemma}
\begin{proof}
The function $ a \to \Gamma^{\pm}(a)$ satisfies
\begin{equation*}
\left\lbrace
\begin{array}{ll}
(\Gamma^{\pm})'(a)=[\mu+\lambda^{\pm}]\Gamma^{\pm}(a)-\alpha  \delta_\pm \beta(a),\ \mbox{for a. e.}\ a\geq 0,\\
\Gamma^{\pm}(0)=1.
\end{array}
\right.
\end{equation*}
By using classical solutions of \eqref{3.1}, we deduce that
$$
 \frac{\dd}{\dd t}\int_0^{\infty}\Gamma^{\pm}(a) u_{\pm}(t,a) \dd a= \int_0^{\infty}\Gamma^{\pm}(a)[-\partial_a u_{\pm}(t,a)-\mu u_{\pm}(t,a)] \dd a.
$$
{Using} integration by parts, it follows that
\begin{eqnarray*}
 \frac{\dd}{\dd t}\int_0^{\infty}\Gamma^{\pm}(a)u_{\pm}(t,a) \dd a
&=&\Gamma^{\pm}(0) u_{\pm}(t,0)+\int_0^{\infty}(\Gamma^{\pm})'(a)u_{\pm}(t,a) \dd a\\
& & - \int_0^{\infty}\Gamma^{\pm}(a)\mu u_{\pm}(t,a) \dd a.
\end{eqnarray*}
Using the facts $\Gamma^{\pm}(0)=1$ and $u_{\pm}(t,0)=\alpha \delta_{\pm}\int_0^{+\infty} \beta(a)u_{\pm}(t,a)\dd a$, we have
\begin{equation*}
\frac{\dd}{\dd t}\int_0^{\infty}\Gamma^{\pm}(a)u_{\pm}(t,a) \dd a=\lambda^{\pm}\int_0^{\infty}\Gamma^{\pm}(a)u_{\pm}(t,a) \dd a,
\end{equation*}
{and} the result follows {from} the fact that the set of initial values giving a classical solution is dense in $L^1$.
\end{proof}

By Assumption \ref{ASS1.1} we have $\beta \neq 0$. Therefore we can define
$$
a^\star:=\sup \left\{a>0: \int_a^{\infty}\beta(\sigma) e^{- \mu \sigma} \dd\sigma >0  \right\} \in (0, \infty].
$$
\begin{remark} \label{REM3.4}
In practice, the number $a^\star$ corresponds to the maximal value at which individuals can reproduce.

Assume for simplicity that $a \to \beta(a)$ is a continuous function. Then, $a^\star=\infty$ if and only if for each $a \geq 0$, there exists $\widehat{a}>a$ such that
$
\beta(\widehat{a})  >0.
$
If $a^\star<\infty$, then
$$
\beta(a)=0, \forall a \geq a^\star.
$$
\end{remark}
Define the interior sub-domain
$$
\widehat{M}_0=\left \{u \in \Li^1_+(0,+\infty): \int_0^{a^\star} u(a)\dd a>0 \right\},
$$
and the boundary sub-domain
$$
\partial \widehat{M}_0=\left \{u \in \Li^1_+(0,+\infty): \int_0^{a^\star} u(a)\dd a = 0 \right\}.
$$
Actually, the boundary domain $\partial \widehat{M}_0$ corresponds to a case where the distribution $u$ contains only individuals that will not reproduce. However, some individuals in the interior region will produce  {newborn} individuals. Due to the irreducible structured of the semiflow generated by \eqref{1.1}, we can obtain the invariance of $\partial \widehat{M}_0$ and $\widehat{M}_0$.

\begin{theorem}\label{TH3.5}
Let Assumption \ref{ASS1.1} be satisfied. The domains $\left[ 0,\infty \right) \times \widehat{M}_0 $ and  $\left[ 0,\infty \right) \times  \partial \widehat{M}_0 $  are positively invariant by the semiflow generated by \eqref{1.1}. {That is} to say that
$$
\int_0^{a^\star} u_0(a)\dd a >0 \Rightarrow \int_0^{a^\star} u(t,a)\dd a > 0, \forall t \geq 0 ,
$$
and
$$
\int_0^{a^\star} u_0(a)\dd a=0 \Rightarrow \int_0^{a^\star} u(t,a)\dd a=0, \forall t \geq 0.
$$
Moreover if $\int_0^{a^\star} u_0(a)\dd a=0$, then
$$
\int_0^{\infty} \beta(a)u(t,a)\dd a=0, \forall t \geq 0 ,
$$
and the solution is explicitly given by
\begin{equation*}
u(t,a)=
\left\lbrace
\begin{array}{l}
e^{-\mu t} u_0(a-t), \text{{\rm if }} a-t \geq 0,\\
0, \text{{\rm if }} t-a \geq 0.
\end{array}
\right.
\end{equation*}
Therefore,
$
\lim_{t \to \infty} \Vert u(t,.) \Vert_{\Li^1}=0,
$
and the convergence is exponential.
\end{theorem}
\begin{proof}Assume first $u_0 \in \partial \widehat{M}_0$. Then,
$
\int_0^{\infty}\Gamma^{+}(a)u_0(a) \dd a=0.
$
Moreover,
$$
\int_0^{\infty}\Gamma^{+}(a)u(t,a) \dd a \leq \int_0^{\infty}\Gamma^{+}(a)  u_+(t,a) \dd a =e^{\lambda^{+}t}\int_0^{\infty}\Gamma^{+}(a)u_0(a) \dd a=0  {.}
$$
Therefore,
$
\int_0^{\infty}\Gamma^{+}(a)u(t,a) \dd a =0, \forall t \geq 0.
$
Hence,
$$
u(t,.) \in \partial \widehat{M}_0, \forall t \geq 0.
$$

Assume next that $u_0 \in  \widehat{M}_0$. Then it follows that
$
\int_0^{\infty}\Gamma^{-}(a)u_0(a) \dd a>0,
$
and
$$
\int_0^{\infty}\Gamma^{-}(a)u(t,a) \dd a \geq \int_0^{\infty}\Gamma^{-}(a) u_- (t,a) \dd a =e^{\lambda^{-}t}\int_0^{\infty}\Gamma^{-}(a)u_0(a) \dd a>0.
$$
Thus,
 $$
u(t,.) \in \widehat{M}_0, \forall t \geq 0.
$$
\end{proof}

Let $u_0 \in   \widehat{M}_0$. Since the function $a \to \beta(a)$ is not assumed to be strictly positive on the interval $[0,a^*)$, the previous theorem
does not imply that
$$
b(t)=\int_0^{\infty} \beta(a)u(t,a)\dd a>0, \forall t \geq 0.
$$
By using the previous assumption, we should consider two different {cases}: 1) $a^*=\infty$, then
$$
\int_a^\infty \beta(a)e^{-\mu a}\dd a >0, \forall a >0;
$$
and 2)  $a^*<\infty$, then
$$
\int_a^{a^*} \beta(a)e^{-\mu a}\dd a >0, \forall a \in (0,a^*).
$$

\begin{proposition}\label{PRO3.6}
 Let Assumption \ref{ASS1.1} be satisfied. For each $u_0 \in \widehat{M}_0$, there exists $t_0=t_0(u_0)>0$ such that
$
b(t)>0, \forall t \geq t_0.
$
\end{proposition}
\begin{proof} Due to the fact that $u(t,.) \geq u_-(t,.), \forall t \geq 0$, it is sufficient to consider that $t \to B_{-}(t):=\int_0^{+\infty} \beta(a)u_{-}(t,a)\dd a$ is the unique solution of the Volterra integral equation
$$
B_{-}(t)=F_{-}(t)+\int_0^t \gamma(a) B_{-}(t-a) \dd a,
$$
where
$
\gamma(a):=\beta(a)e^{-\mu a}
$
and
\begin{equation*}%\label{2.8}
F_{-}(t)=e^{-\mu t}  \int_t^\infty \beta(a) u_{0}(a-t) \dd a=e^{-\mu t}  \int_0^\infty \beta(\sigma+t) u_{0}(\sigma ) \dd \sigma.
\end{equation*}
By using the continuity of the shift operator in $L^1$, we deduce that the map
$$
t \to \int_t^\infty \beta(a) u_{0}(a-t) \dd a
$$
is continuous. Therefore, the map $t \to F_{-}(t)$ is continuous.

\noindent  \textit{First part of the proof:} Let $\Delta t>0$. By using Fubini's theorem, we have
$$
\int_t^{t+\Delta t}\int_0^\infty \beta(\sigma+\theta) u_{0}(\sigma ) \dd \sigma \dd \theta=\int_0^\infty  \int_t^{t+\Delta t} \beta(\sigma+\theta)  \dd \theta  u_{0}(\sigma )\dd \sigma .
$$
Then, by using the definition of $a^*$, we deduce that there exists $t^* \in (0,a^*)$ such that
\begin{equation*}%\label{2.9}
\int_0^\infty \beta(\sigma+t^*) u_{0}(\sigma ) \dd \sigma >0 \Leftrightarrow F_{-}(t^*)>0.
\end{equation*}

\noindent \textit{Second part of the proof: } In order to prove this theorem, we observe that for all $t \geq 0$
$$
\int_t^{t+\Delta t} B_{-}(\theta) \dd \theta =\int_t^{t+\Delta t} F_{-}(\theta) \dd \theta + \int_t^{t+\Delta t} \int_0^t \gamma(\theta-a) B_{-}(a) \dd a \dd \theta.
$$
 Then, by using Fubini's theorem, we obtain for all $t \geq 0$
$$
\int_t^{t+\Delta t} B_{-}(l) \dd l =\int_t^{t+\Delta t} F_{-}(l) \dd l +  \int_0^t \int_{t-a}^{t-a+\Delta t} \gamma(\theta) \dd \theta  B_{-}(a)  \dd a.
$$
Now, you can work like if $\gamma(q)$ were a continuous function, and we deduce that the length of the support $t \to B_{-}(t)$ increases from $(0,a^*)$ to $(a^*,2a^*)$. By induction, we deduce that there exists an integer $m\geq 1$ such that $t \to B_{-}(t)$ is strictly positive on $(ma^*,(m+1)a^*)$. The result follows.
\end{proof}

\section{Asymptotic smoothness of the semiflow}
\label{section4}
In this section, we prove that the semiflow generated by \eqref{1.1} is asymptotically smooth. Define
\begin{equation*}%\label{3.1}
  u_1(t, a) =
            \left\{
                  \begin{array}{cc}
                     e^{-\mu t} u_0(a - t),  & \text{ if }  a \geq t,\\
                    e^{-\mu a} b_1(t -a), &  \text{ if }   a \leq  t,
                  \end{array}
            \right.
\end{equation*}
and
\begin{equation*}%\label{3.2}
  u_2(t, a) =
            \left\{
                  \begin{array}{cc}
                    0,  & \text{ if }  a \geq t,\\
                    e^{-\mu a} b_2(t -a), &  \text{ if }   a \leq  t,
                  \end{array}
            \right.
\end{equation*}
where
\begin{equation*} % \label{3.3}
\begin{array}{l}
b_1(t)=\alpha \left[ f \left(B(t) \right)-f(B_2(t)) \right], \quad
b_2(t)= \alpha f(B_2(t)),
\end{array}
\end{equation*}
and
\begin{equation*}  %\label{3.4}
B(t)=\int_0^\infty u(t,a)\dd a=\underset{B_1(t)}{\underbrace{e^{-\mu t} \int_t^\infty \beta(a)u_0(a-t)\dd a}}+ \underset{B_2(t)}{\underbrace{ \int_0^t \beta(a) e^{-\mu a} b(t-a)\dd a}}.
\end{equation*}

\begin{theorem}
The semiflow generated by \eqref{1.1} is asymptotically smooth. Moreover, precisely the semiflow $U(t)$ generated by \eqref{1.1} on $L^1_+$  can be decomposed into
$$
U(t)u_0=C(t)u_0+V(t)u_0, \forall t \geq 0, \forall u_0 \in L^1_+,
$$
where $C(t):L^1_+ \to L^1$ and $V(t):L^1_+ \to L^1$  are defined by
$$
C(t)u_0=u_2(t,.), \text{ and } V(t)u_0=u_1(t,.),  \forall t \geq 0, \forall u_0 \in L^1_+.
$$
Then we have the following properties
\begin{itemize}
\item[{\rm (i)}]  The nonlinear operator $C(t):L^1_+ \to L^1$ is completely continuous for each $t \geq 0$;
\item[{\rm (ii)}] There exists a constant $\chi>0$ such that
$
\Vert V(t)u_0 \Vert \leq \chi e^{-\mu t }, \forall t \geq 0.
$
\end{itemize}
\end{theorem}
\begin{proof}
\textit{Proof of (i):} It is sufficient to observe that $t \to B_2(t)$ is uniformly continuous on $[0, \tau]$ uniformly with respect to $b(t)$. Indeed, we have  $\forall t,s \in [0, \tau]$ (with $t\geq s$)
$$
\vert B_2(t)-B_2(s) \vert \leq \sup_{\sigma \in [0,\tau]} \vert b(\sigma)\vert \left[\int_0^{t-s}g(l) \dd l+ \int_0^s \vert g(t-s+l)-g(l)\vert \dd l  \right],
$$
where $g(a)=\beta(a) e^{-\mu a}$. By applying Arzela-Ascoli's theorem, it follows that $C(t)$ maps the bounded sets into relatively compact sets.

\noindent \textit{Proof of (ii):} Moreover, since $f(u)$ is Lipschitz continuous on $[0, \infty)$,  we deduce that
$$
\Vert u_1(t,.) \Vert_{L^1} \leq e^{-\mu t}  \Vert  u_0 \Vert_{L^1}  +  \Vert  e^{-\mu . }\Vert_{L^1} \alpha \Vert f' \mid_{[0, \infty)} \Vert_{\infty}  B_1(t).
$$
Thus,  there exists a constant $\chi=1+\Vert  e^{-\mu . }\Vert_{L^1} \alpha \Vert f' \mid_{[0, \infty)} \Vert_{\infty} \Vert \beta  \Vert_{L^\infty}$ such that
$$
\Vert u_1(t,.) \Vert_{L^1} \leq \chi e^{-\mu t}  \Vert  u_0 \Vert_{L^1}, \forall t \geq 0.
$$
The proof is completed.
\end{proof}
\begin{theorem}
The semiflow generated by \eqref{1.1} is asymptotically smooth. That is to say that
$$
 \lim_{t \to \infty } \kappa \left( U(t)B \right) =0,
$$
where $\kappa$ is the Kuratovsky's measure of non-compactness defined by
$$
\kappa (B)=\inf \left\lbrace \varepsilon>0: B \text{ can be covered by a finite number of balls of radius } \leq  \varepsilon \right\rbrace,
$$
for any bounded set $B$ of $L^1_+(0, \infty)$.
\end{theorem}
\begin{proof}
For each bounded set $B \subset L^1_+(0, \infty)$, we have the Kuratovsky's measure of non-compactness satisfies
$$
\kappa( U(t)B) \leq \kappa( V(t)B) +\kappa( C(t)B) =\kappa( V(t)B) \leq \chi e^{-\mu t}  \sup_{u \in B} \Vert u \Vert_{L^1}.
$$
The proof is completed.
\end{proof}
\section{Uniform persistence and global attractor}
\label{section5}
In this section, motivated by the Magal \cite{Magal09}, Magal and Zhao \cite{Magal-Zhao}, Zhao \cite{Zhao-2013-dynamical}, we obtain the existence of the global attractor.

\begin{assumption}\label{ASS5.1}
   We assume that $  \alpha  > 1$. %, where $\delta_* = e^{-\varepsilon \Vert \beta \Vert_{L^\infty} }$.
\end{assumption}

\begin{theorem}
   Let  Assumptions \ref{ASS1.1} and \ref{ASS5.1}   be satisfied.
   Then the semiflow $U(t)$ is uniformly persistent in $\widehat{M}_0$.
   That is, there exists $\varepsilon > 0$ such that for any $u_0 \in \widehat{M}_0$,
\begin{equation*}%\label{4.1}
   \liminf_{t \to +\infty} \Vert u(t, \cdot)  \Vert_{L^1 } \geq \varepsilon.
\end{equation*}
Furthermore, the semiflow $U$ has a compact global attractor $\mathcal{A}_0$ in $\widehat{M}_0$.
\end{theorem}
\begin{proof}
 %  It follows from Theorem \ref{TH3.5}that the extinction equilibrium $ u =0$ is globally exponentially stable in $\partial \widehat{M}_0$.

Assume that the semiflow $U(t)$ is not uniformly persistent in $\widehat{M}_0$.
Then for
any  $\varepsilon > 0$,  there exists some $u_0 \in \widehat{M}_0$ such that
\begin{equation}\label{5.1}
    \Vert u(t, \cdot)  \Vert_{L^1} \leq \varepsilon, \quad \forall t\geq 0.
\end{equation}
Thus, we have
\begin{equation*}
   \exp \left( {-\int_{0}^{+\infty} \beta(a) u(t, a) \dd a } \right) \geq \exp \left( { - \Vert \beta \Vert_{L^{\infty}} \Vert u(t, \cdot) \Vert_{L^1} } \right)  \geq  \delta_*,
\end{equation*}
where $\delta_* = e^{-\varepsilon \Vert \beta \Vert_{L^\infty} }$.
Thanks to the comparison principle of Theorem \ref{TH3.1}, we have
\begin{equation*}
   u(t, a) \geq u_*(t, a), \quad t \geq 0, \mbox{for a. e.}\ a\geq 0,
\end{equation*}
where $u_*(t, a)$ satisfies
\begin{equation*}%\label{4.3}
  \left\lbrace
\begin{array}{l}
    \partial_t u_{*}(t,a)+\partial_a u_{*}(t,a)=-\mu u_{*}(t,a), \text{ for } a \geq 0,\ t\geq0,\\
    u_{*}(t,0)=\alpha \delta_* \int_0^{+\infty} \beta(a)u_{*}(t,a)\dd a, \\
    u_{*}(0,.)=u_0 \in  \widehat{M}_0.
\end{array}
\right.
\end{equation*}
For this $u_0\in \widehat{M}_0$, we have  $ \int_0^{\infty}\Gamma^{-}(a)u_0(a) \dd a>0 $ and
\begin{equation*}
 \int_0^{\infty}\Gamma^{-}(a)u(t,a) \dd a \geq \int_0^{\infty}\Gamma^{-}(a)u_*(t,a) \dd a =e^{ \lambda_{*} t}\int_0^{\infty}\Gamma^{-}(a)u_0(a) \dd a>0,
\end{equation*}
where $\lambda_* \in \mathbb{R}$ satisfies
\begin{equation*}
   \alpha \delta_* \int_{0}^{+\infty } \beta(a) e^{- ( \mu + \lambda_* ) a } \dd a = 1.
\end{equation*}
By Assumption \ref{ASS5.1}, we can choose $\varepsilon> 0$ small enough to ensure that $\alpha \delta_* > 1$.
Therefore, it follows from the definition of the $\lambda_*$ and Assumption \ref{ASS1.1} that $\lambda_* > 0$.
Thus,
\begin{equation*}
  \lim_{t\to +\infty } \Vert u(t, \cdot) \Vert_{L^1} \geq \lim_{t\to \infty } \int_{0}^{\infty } \Gamma^-(a) u(t, a ) \dd a = +\infty,
\end{equation*}
which  contradicts to \eqref{5.1}.

Since the semiflow $U$ is asymptotically smooth, point dissipative and positive orbits of compact subsets of $\widehat{M}_0$ are bounded, we deduce that $U$ has a compact global attractor $\mathcal{A}_0 $ in $\widehat{M}_0 $.
\end{proof}

\subsection{Complete orbits}

\begin{proposition}\label{PRO5.3}
   Let Assumption \ref{ASS1.1} be satisfied. There exists $\delta > 0$ such that
\begin{equation*}
   \int_{0}^{+\infty} \beta(a) u_0(a) \dd a \geq \delta , \quad \forall u_0 \in \mathcal{A}_0.
\end{equation*}
\end{proposition}
\begin{proof}
   Let $u(t, \cdot)$ be the solution to \eqref{1.1} with the initial condition $u_0 \in \mathcal{A}_0$.
%By comparison principle of Theorem \ref{TH3.1}, we have
%\begin{equation*}
%   u(t, a) \geq u_-(t, a), \quad t \geq 0, \mbox{for a. e.}\ a\geq 0,
%\end{equation*}
%where $u_-(t, a)$ satisfies
%\begin{equation}\label{eq1}
%  \left\lbrace
%\begin{array}{l}
%    \partial_t u_{-}(t,a)+\partial_a u_{-}(t,a)=-\mu u_{-}(t,a), \text{ for } a \geq 0,\ t\geq0,\\
%    u_{-}(t,0)=\alpha \delta_* \int_0^{+\infty} \beta(a)u_{-}(t,a)da, \\
%    u_{-}(0,.)=u_0 \in  \mathcal{A}_0.
%\end{array}
%\right.
%\end{equation}
By Proposition \ref{PRO3.6}, there exists $t_0 = t_0(u_0)> 0$ such that
\begin{equation*}
   \int_{0}^{+\infty} \beta(a) u(t, a) \dd a > 0, \quad \forall t \geq t_0.
\end{equation*}
Recall that the mapping $(t, u_0) \to U(t)u_0 $ is continuous.
Also, the mapping $u_0 \to \int_{0}^{+\infty} \beta(a) u_0(a) \dd a$ is continuous.
Thus, there exists $\varepsilon = \varepsilon(u_0)> 0$ such that if $\tilde{u}_0 \in \mathcal{A}_0$ with
$
   \Vert \tilde{u}_0 - u_0  \Vert \leq \varepsilon,
$
and the solution to \eqref{1.1} with the initial condition $\tilde{u}_0 $ is denoted by $\tilde{u}(t, \cdot)$, then
\begin{equation*}
   \int_{0}^{+\infty } \beta(a) \tilde{u}(t, a) \dd a > 0, \quad \forall t \in [t_0, t_0 + a_0],
\end{equation*}
for some $a_0 > 0$. Thus, it follows from the proof of the Proposition \ref{PRO3.6} that
\begin{equation*}
   \int_{0}^{+\infty } \beta(a) \tilde{u}(t, a) \dd a > 0, \quad \forall t \geq t_0.
\end{equation*}
Therefore, by using the compactness of $\mathcal{A}_0$, we deduce that there exists $\hat{t}> 0$ (independent of $\mathcal{A}_0$) such that for each $u_0 \in \mathcal{A}_0$,
\begin{equation*}
   \int_{0}^{+\infty} \beta(a) u(t, a) \dd a > 0, \quad \forall t \geq \hat{t}.
\end{equation*}
By using the fact that $\mathcal{A}_0$ is invariant under $U$, it follows that for each $u_0\in \mathcal{A}_0$,
\begin{equation*}
   \int_{0}^{+\infty} \beta(a) u_0(a) \dd a > 0.
\end{equation*}
Now, by  using the continuity of $u_0 \to \int_{0}^{+\infty} \beta(a) u_0(a) \dd a $ again and the compactness of $\mathcal{A}_0$, we obtain  the results and complete the proof.
\end{proof}

\begin{theorem}
Assume that $t \to u(t,a)$ is a complete orbit of the system \eqref{1.1}. Then
\begin{equation} \label{5.2}
u(t,a)=e^{-\mu a} b(t-a), \text{ for almost every } a \geq 0, \forall t\in \mathbb{R},
\end{equation}
and the map $t \to b(t)$ satisfies the renewal equation
\begin{equation} \label{5.3}
b(t)=\alpha f\left(  \int_{0}^{\infty} \beta(a) e^{-\mu a} b(t-a) \dd a\right), \forall t\in \mathbb{R}.
\end{equation}
Moreover, if   we assume in addition that
\begin{equation} \label{5.4}
\gamma_-  \leq \int_0^{a_*}u(t,a)\dd a \leq \gamma_+, \forall t\in \mathbb{R},
\end{equation}
for some $\gamma_+>\gamma_->0$, then
\begin{equation*} %\label{4.7}
 \delta^- \leq \int_0^\infty \beta(a)u(t,a)\dd a \leq \delta^+, \forall t\in \mathbb{R},
\end{equation*}
for some $\delta^+  \geq \delta^->0$.
\end{theorem}
\begin{proof}
Let $t_0 \leq 0$ and $t \geq t_0$.
Due to $t \to u(t, a)$ is the  complete orbit of the system \eqref{1.1}, the solution integrated along the characteristic of \eqref{1.1} with the initial value $u(t_0, a)$ is
\begin{equation}\label{5.5}
  u(l, a) =
            \left\{
                  \begin{array}{ll}
                     e^{-\mu (l - t_0) } u(t_0, a - (l -t_0)),  & \text{ if }  a  \geq  l-  t_0,\\
                    e^{-\mu a} b(l -t_0 -a), &  \text{ if }     a \leq l-t_0,
                  \end{array}
            \right.
\end{equation}
where $l \to b(l)$ is the unique continuous solution of the Volterra's integral equation
\begin{equation}\label{5.6}
   b(l - t_0) = \alpha f\left( \int_{l-t_0}^{\infty } \beta(a) e^{-\mu (l- t_0)} u(t_0, a - (l -t_0)) \dd a+ \int_{0}^{l - t_0} \beta(a) e^{-\mu a} b(l - t_0 - a ) \dd a\right).
\end{equation}
Noticing \eqref{5.5}, \eqref{5.6} and \eqref{5.3} and setting $t = l -t_0$,  the result \eqref{5.2} follows.

It follows from \eqref{5.2} that
\begin{equation*}
   \int_{0}^{+\infty} \beta(a) u(t, a) \dd a  = \int_{0}^{+\infty } \beta(a) e^{-\mu a} b(t - a ) \dd a \leq \delta^+,
\end{equation*}
where $\delta^+ = \alpha \max\limits_{u\geq 0} f(u) $.

Since $\mathcal{A}_0$ is the global attractor, which implies that $U(t) \mathcal{A}_0 = \mathcal{A}_0 $ for all $t\geq 0$, we can find a complete orbit  $u(t, \cdot)$ through $u(0, \cdot)$.
Due to \eqref{5.4}, it follows that for all $t\in \mathbb{R}$, there exists a $\hat{t} \leq t$ such that
$
   \int_{0}^{+\infty} \beta(a) u(\hat{t}, a) \dd a >0.
$
Then, by the method of Proposition \ref{PRO3.6}, we have
\begin{equation*}
   \int_{0}^{+\infty} \beta(a) u(t, a) \dd a > 0, \quad  t \in \mathbb{R}.
\end{equation*}
Since $\mathcal{A}_0$ is compact, it follows that there exists some $\delta^- > 0 $ such that
\begin{equation*}
   \int_{0}^{+\infty} \beta(a) u( t, a) \dd a \geq \delta^-,
\end{equation*}
which completes the proof.
\end{proof}

\section{Global stability result}
\label{section6}
In this section, we use a Lyapunov functional to show that  all solutions of \eqref{1.1} for which the population is initially present tend to the positive equilibrium under the Assumption \ref{ASS5.2}.

Define
\begin{equation}\label{6.1}
   V(u(t, \cdot) ) =  \int_{0 }^{ +\infty } \gamma(a) g\left( \frac{ u(t, a)}{ \bar{u} (a) } \right) \dd a,
\end{equation}
where
\begin{equation*}%\label{5.2}
   \gamma(a) = \int_{a}^{+\infty} \beta(\sigma) e^{-\mu \sigma} \dd \sigma,
\quad
   g( u ) = (u -1 )^2.
\end{equation*}

\begin{remark}
   We observe that  $ V$ is well-defined on the attractor $\mathcal{A}_0$ for every $t\geq 0$, while $V$ is not well-defined on $L^1(0, +\infty)$ because of the function $g$ under the integral.
\end{remark}
%\begin{remark} It is not clear that know if the Lyapunov function
%$$
%V(u)= \int_{0 }^{ +\infty } \gamma(a) g\left( \frac{ u(a)}{ \bar{u}  (a) } \right) \dd a,
%$$
% is not defined or not for each $u \in \Li^1(0,\infty)$. Since
%$$
%\gamma(a) \leq \Vert \beta \Vert_{\Li^\infty} e^{-\mu a }/\mu.
%$$
% So in order to get a Lyapunov function that well defined we will work initial distribution $u_0$ the interior global attractor
%$$
%u(a) \leq C e^{-\mu a }
%$$
%for some constant $C>0$.
%\end{remark}
\begin{assumption} \label{ASS5.2}
    We assume that $1 < \alpha \leq e^{2} $.
\end{assumption}
The following Lemma is due to Fisher and Goh \cite{Fisher-Goh}.  % For completeness, we reconsider their proof.
\begin{lemma} \label{LE5.3}
Let Assumption \ref{ASS5.2} be satisfied. Then for each $u \geq 0$,
\begin{equation}\label{6.2}
   \left\vert{ \alpha f(u) - \overline{u} (0) }  \right\vert= \left\vert{ \alpha f(u) -  \alpha f(\overline{u} (0))} \right\vert   \leq \left\vert{ u - \overline{u} (0) }\right\vert.
\end{equation}
Moreover the above inequality is strict whenever $u\neq 0$ and $u \neq  \overline{u} (0)$.
\end{lemma}

\begin{proposition}[Lyapunov function] \label{PRO5.4}
Let Assumptions \ref{ASS1.1} and \ref{ASS5.2} be satisfied. Assume that $t \to u(t,a)$ is a complete orbit of the system \eqref{1.1} satisfying
\begin{equation*}
\gamma_-  \leq \int_0^{a_*}u(t,a)\dd a \leq \gamma_+, \forall t\in \mathbb{R},
\end{equation*}
for some $\gamma_+>\gamma_->0$.
Then the function $t \to V(u(t, \cdot) )$ is decreasing. Moreover, $t \to V(u(t, \cdot) )$ is constant if
$
u(t,a)=\overline{u}(a), \forall t \in \R.
$
\end{proposition}
\begin{proof} We have
\begin{align*}
   \frac{\dd V(u(t, \cdot) ) }{\dd t}
   = & \frac{ \dd }{ \dd  t} \int_{0 }^{ +\infty } \gamma(a) g\left( \frac{ u(t, a)}{ \bar{u}  (a) } \right) \dd a
   =   \frac{\dd }{ \dd  t} \int_{0}^{+\infty } \gamma(a) g \left( \frac{ b(t-a ) }{ \bar{u}  (0) } \right) \dd a  \\
   = & \frac{\dd }{\dd  t} \int_{-\infty }^{ t}  \gamma(t-s) g\left( \frac{b(s)}{ \bar{u}  (0) } \right) \dd s \\
   = & \int_{-\infty }^{ t} \gamma'(t-s ) g\left(  \frac{b(s) }{ \bar{u} (0) } \right) \dd s + \gamma(0 ) g\left( \frac{b(t)}{ \bar{u}  (0) } \right).
\end{align*}
It follows from the definition of $\gamma(a)$ that $ \gamma'(a) = - \beta(a) e^{-\mu a}$ and $\gamma(0) = 1$.
Then,
\begin{equation}\label{6.3}
   \frac{\dd V(u(t, \cdot  ) )}{ \dd t} =  g\left( \frac{b(t)}{ \bar{u}  (0) } \right) - \int_{0}^{+\infty} \beta(a) e^{-\mu a} g\left( \frac{b(t -a) }{ \bar{u} (0) } \right) \dd a.
\end{equation}
By using Lemma \ref{LE5.3},  we have
\begin{equation} \label{6.4}
\begin{array}{ll}
 g\left( \dfrac{b(t)}{ \bar{u}  (0) } \right)
    =  &  \dfrac{1}{ \bar{u}^2  (0) } \left[ \alpha f\xk{ \displaystyle \int_{0}^{+\infty } \beta(a) e^{-\mu a} b(t-a) \dd a } - \bar{u} (0) \right]^2  \\
 & \leq  \dfrac{1}{ \bar{u}^2 (0) } \left[{ \displaystyle \int_{0}^{+\infty } \beta(a) e^{-\mu a} b(t-a) \dd a
            - \bar{u} (0) }\right]^2.
\end{array}
\end{equation}
Thus,
\begin{equation} \label{6.5}
g\left( \frac{b(t)}{ \bar{u} (0) } \right) \leq  g\xk{ \int_{0}^{+\infty } \beta(a) e^{-\mu a} \frac{ b(t-a)}{  \bar{u} (0) } \dd a }.
\end{equation}
Substituting the inequality \eqref{6.5} to \eqref{6.3}, we obtain
\begin{align*}
   \frac{\dd V(u(t, \cdot)) }{ \dd t}
   \leq & g\xk{ \int_{0}^{+\infty } \beta(a) e^{-\mu a} \frac{ b(t-a)}{  \bar{u} (0) } \dd a }
          - \int_{0}^{+\infty} \beta(a) e^{-\mu a} g\xk{ \frac{b(t -a) }{ \bar{u}  (0) } } \dd a,
\end{align*}
and we deduce that
$
\frac{\dd V(u(t, \cdot)) }{ \dd t} \leq 0
$
by using Jensen's inequality and the fact that $g$ is a convex function.

To prove the last part of the proposition,  it is sufficient to observe that by Lemma \ref{LE5.3},  the inequality   \eqref{6.4} is strict whenever $\int_{0}^{+\infty } \beta(a) e^{-\mu a} b(t-a) \dd a \neq \overline{u}(0)$. Therefore,  the function $t \to  V(u(t, \cdot)) $ is constant if and only if
$$
\int_{0}^{+\infty } \beta(a) e^{-\mu a} b(t-a) \dd a = \overline{u}(0), \forall t \in \R,
$$
and the result follows by using \eqref{5.2} and \eqref{5.3}.
\end{proof}

\begin{theorem}\label{TH-GAS}
   Let Assumptions \ref{ASS1.1} and \ref{ASS5.2} be satisfied. Then, the positive equilibrium $\bar{u}(a) $ of the system \eqref{1.1} is global asymptotically stable.
\end{theorem}
\begin{proof}
   Let $u(t, \cdot)$ be a complete solution to \eqref{1.1} that lies in the attractor $\mathcal{A}_0 \subset \widehat{M}_0$.
   It follows from Proposition \ref{PRO5.3} that there exist $\delta_1, \delta_2> 0$ such that
\begin{equation}\label{6.6}
  \delta_1 \leq \frac{u(t, a)}{ \bar{u}(a) } \leq \delta_2, \quad \forall a\geq 0, t\in \mathbb{R}.
\end{equation}
By the definition of the function $g$ in \eqref{6.1}, we observe that $g:(0, +\infty) \to \mathbb{R}_+$ is decreasing for $u \in (0, 1)$ and is increasing for $u > 1$ and has global minimum $g(1) = 0$.
It follows from \eqref{6.6} that $g \left( \frac{u(t, a)}{\bar{ u }(a) } \right)$ is bounded.
Thus, the Lyapunov functional $V(t)$ in \eqref{6.1} is bounded for all $t\in \mathbb{R}$.
Furthermore, by Proposition \ref{PRO5.4},   the function $t \to V(u(t, \cdot) )$ is decreasing and    is constant if
$
u(t,a)=\overline{u}(a), \quad \forall t \in \R.
$
Thus, we have now shown that the arbitrary complete solution $u(t, \cdot)$ in the attractor $\mathcal{A}_0$ must be the positive equilibrium $\bar{u}(a)$, namely $\mathcal{A}_0 = \{ \bar{u}(a) \}$.
\end{proof}

\section{Hopf's bifurcation}
\label{section7}
This section is based on Magal and Ruan \cite{Magal-Ruan2009}.
In order  to study Hopf bifurcation, we should consider the parameter $\alpha $ in the region $\{ \alpha \in \mathbb{R}: \alpha > e^2  \}$. The linearized system \eqref{1.1} around $\bar{u}(a)$ is
\begin{equation}\label{7.1}
\left\{
\begin{array}{ll}
(\partial_t + \partial_a  ) u(t, a ) = - \mu u(t, a), & t> 0, a> 0,    \\
u(t, 0) = (1 - \ln\alpha )  \displaystyle \int_{0}^{\infty} \beta (a) u(t, a) \dd a , & t> 0.
\end{array}
\right.
\end{equation}
Then,   the characteristic equation of the linear system \eqref{7.1} can be expressed by
\begin{equation}\label{7.2}
\Delta(\lambda, \alpha) = 1 - ( 1 -\ln \alpha ) \int_{0}^{+\infty} \beta(a) e^{ -( \lambda + \mu  )  } \dd a = 0.
\end{equation}
%{ \color{red}
%Here, we need the specific expression for the birth function $\beta(a)$.
%
%\begin{equation*}
%   \beta (a) = (a -\tau )^{n } e^{-\kappa (a - \tau )}1_{[\tau, +\infty}(a),
%\end{equation*}
%or some function else ?
%\begin{equation*}
%   \beta(a) = \beta^* 1_{[\tau, +\infty}(a).
%\end{equation*}
%}
%

\begin{assumption}\label{ASS7.1}
	Assume the birth function
	\begin{equation*}
	\beta (a) = C_0(a -\tau )^{n } e^{-\kappa (a - \tau )}1_{[\tau, +\infty) }(a),
	\end{equation*}
	with $\tau > 0, \kappa > 0, n \in \mathbb{N} $ or $ \tau> 0, \kappa = 0, n = 0 $, where the constant $C_0$ is dependent on the parameters $\tau, \kappa$ and $n$ to ensure that $\int_{0}^{+\infty} \beta(a) e^{-\mu a} \dd a = 1 $.
\end{assumption}

Under Assumption \ref{ASS7.1}, the characteristic equation \eqref{7.2} can be rewritten as
\begin{equation}\label{7.3}
\Delta(\lambda, \alpha) = 1 - ( 1 -\ln \alpha )   e^{ -\lambda \tau } \left( 1+ \frac{\lambda}{\mu + \kappa } \right)^{ - n - 1} = 0.
\end{equation}

Now we are position to look for purely imaginary roots $\lambda = \pm \mathrm{i} \omega$ with $\omega > 0$ of the characteristic equation \eqref{7.3}.

\begin{proposition}\label{PRO7.2}
	Let Assumption \ref{ASS7.1} be satisfied and the parameters $\mu> 0, \tau> 0, \kappa> 0, n\in \mathbb{N}$ be fixed.
	Then the characteristic equation \eqref{7.3} has a  pair of purely imaginary solution $\pm \mathrm{i } \omega $ with $\omega  > 0$ if and only if $\omega>0$ is a solution of
	\begin{equation}\label{7.4}
	\omega \tau + (n + 1) \arctan \frac{\omega }{ \mu  + \kappa} = \pi + 2 k \pi , \quad k \in \mathbb{N},
	\end{equation}
	 and
	\begin{equation}\label{7.5}
	\alpha  = \exp \left(  1 + \left( 1 + \frac{\omega^2 }{ ( \mu + \kappa)^2 } \right)^{ \frac{ n +1}{ 2} } \right).
	\end{equation}
	Furthermore, for each $k \in \mathbb{N}$, there exists exactly one solution $\omega_k $ for equation \eqref{7.4}, i.e., the characteristic equation \eqref{7.3} has exactly one pair of purely imaginary eigenvalues $\pm \mathrm{i} \omega_k$ with $\omega_k> 0$ for each
	\begin{equation}\label{7.6}
	\alpha_k = \exp \left(  1 + \left( 1 + \frac{\omega_k^2 }{ ( \mu + \kappa)^2 } \right)^{ \frac{ n +1}{ 2} } \right).
	\end{equation}
	Moreover,
	$
	\omega_k \to +\infty  \text{ and } \alpha_k \to + \infty, \text{ as } k\to +\infty.
	$
	If $\omega_k(n)> 0$ is the solution of equation \eqref{7.4} with fixed  $k  \in \mathbb{N}$ for any $n \in \mathbb{N}$, then
	\begin{equation}\label{7.7}
	\omega_k(n) \to 0 \text{ and }   \alpha_k( n ) \to e^2 , \text{ as } n\to +\infty.
	\end{equation}
%	and
%	\begin{equation*}
%	\alpha_k( n ) \to e^2, \text{ as } n\to +\infty.
%	\end{equation*}
\end{proposition}
\begin{proof}
	Under Assumption \ref{ASS7.1}, if the characteristic equation \eqref{7.3} admits a pair of purely imaginary solution $ \lambda =\pm \mathrm{i} \omega$ with $\omega> 0$, then  \eqref{7.3} can be expressed
	\begin{equation}\label{7.8}
	( 1 - \ln \alpha ) ( r(\omega) )^{-n - 1} e^{ -\mathrm{i}[ \omega \tau +  (n + 1)\theta(\omega) ] } = 1,
	\end{equation}
	where
	\begin{equation*}
	r(\omega) = \sqrt{ 1 + \frac{\omega^2 }{ (\mu + \kappa )^2 } }, \quad \theta (\omega)  = \arctan \frac{\omega }{ \mu + \kappa }.
	\end{equation*}
	Therefore, by  separating the real and imaginary parts of \eqref{7.8}, we can obtain  that $\omega$ satisfies \eqref{7.4}--\eqref{7.5}.
    Thus, the characteristic equation \eqref{7.3} has   a pair of purely imaginary solution $\lambda =\pm \mathrm{i } \omega$ with $\omega> 0$ if and only if $\omega$ is a solution of  \eqref{7.4} and \eqref{7.5}.
	
	Let $h(\omega) = \omega \tau + (n + 1) \arctan \frac{\omega }{ \mu + \kappa }$ for $\omega \geq 0$. Then,    $h(w)$ is strictly increasing   function for $\omega > 0$.
	Noticing  $h(0) = 0$ and $h(\omega ) \to +\infty$ as $\omega \to +\infty$,
	it follows from the continuity of $h(\omega)$ that the equation \eqref{7.5} has exactly one  solution $\omega_k$ for each $k\in \mathbb{N}$ with the parameter $\alpha_k$ satisfying \eqref{7.6}.
	Moreover, let $\omega_k> 0( k \in \mathbb{N})$ satisfy
    \begin{equation*}
        	\omega_k \tau + (n + 1) \arctan \frac{\omega_k }{ \mu  + \kappa} = \pi + 2 k \pi.
    \end{equation*}
    It follows from that $\omega_k \to +\infty $ by letting $k \to +\infty$ on both sides of the above equality.
    The result $\alpha_k \to + \infty   \text{ as } k\to +\infty$ can be obtained from \eqref{7.6} and $\omega_k\to +\infty$ as $k \to +\infty$.

	Let $\omega_k(n)> 0$ be the solution of equation \eqref{7.4} with  fixed $k \in \mathbb{N}$ for any $n \in \mathbb{N}$,  i.e.,
	\begin{equation}\label{7.9}
	\omega_k(n) \tau + (n +1)\arctan \frac{\omega_k(n) }{\mu +\kappa} = \pi + 2k\pi.
	\end{equation}
	Thus, we can obtain from \eqref{7.9} that
	\begin{equation}\label{7.10}
	\omega_k(n) \tau + (n + 1) \arctan \frac{\omega_k(n)}{ \mu + \kappa } = \pi + 2k\pi =\omega_k(n +1) \tau + (n + 2) \arctan \frac{\omega_k(n +1)}{ \mu + \kappa }.
	\end{equation}
	It follows from mean value theorem that there exists some $\xi $ between $\frac{\omega_k(n)}{\mu + \kappa}$ and $\frac{\omega_k(n +1)}{\mu + \kappa}$ such that
	\begin{equation}\label{7.11}
	\arctan \frac{\omega_k(n)}{ \mu + \kappa} - \arctan \frac{\omega_k(n +1)}{\mu + \kappa } = \frac{1}{1 + \xi^2 }\frac{1 }{\mu + \kappa} ( \omega_k(n) - \omega_k(n +1) ).
	\end{equation}
	Therefore, by replacing   \eqref{7.11} to \eqref{7.10}, we can rewrite \eqref{7.10} as
	\begin{equation} \label{7.12}
	\left( \tau + \frac{n + 1}{(\mu + \kappa) ( 1 + \xi^2 ) }   \right) ( \omega_k(n) - \omega_k(n +1) )  = \arctan \frac{\omega_k(n +1)}{ \mu + \kappa }.
	\end{equation}
    Thanks to the positivity of  $\omega_k(n)$, %and the parameters $\mu, \kappa$,
    we obtain from \eqref{7.12} that
    $$\omega_k(n ) > \omega_{k}(n + 1)> 0, \text{ for all } n \in \mathbb{N},$$
    which implies that
    $\omega_k(n)$ is strictly decreasing for $n \in \mathbb{N}$.
	Thus, the monotone bounded convergence theorem for the sequence $\{\omega_k(n)\}$  implies that the limit $ \lim\limits_{n \to +\infty} \omega_k(n) $ exists.
Let $\omega^*_k  =\lim\limits_{n \to +\infty} \omega_{k}(n) $. Then we have $\omega_k^* \geq 0$ by the positivity of the sequence $\{\omega_k(n)\}$.
We claim that $\omega_k^* = 0$.
In fact, if $\omega_k^* > 0$, we will obtain a contradiction by letting $n \to +\infty$ on both sides of \eqref{7.9}
Therefore, $\omega_k^* = \lim\limits_{n\to +\infty} \omega_k(n)  = 0$.
By letting $n\to +\infty$ on both sides of equality \eqref{7.9}, we have
    \begin{align*}
      \pi + 2k \pi =    \lim_{n \to +\infty} \left(\omega_k(n) \tau + (n +1)\arctan \frac{\omega_k(n) }{\mu +\kappa}   \right)
       =   \lim_{n \to +\infty}    (n +1)\arctan \frac{\omega_k(n) }{\mu +\kappa},
    \end{align*}
which implies that
    \begin{equation}\label{7.13}
       \omega_k(n) \sim \frac{ (\pi + 2k\pi)  (\mu + \kappa )}{n + 1} \text{ as } n \to +\infty.
    \end{equation}
	Therefore,  the result $ \alpha_k(n ) \to e^2 \text{ as } n\to +\infty$ in \eqref{7.7} can be obtained from   \eqref{7.5}  and \eqref{7.13}.
\end{proof}

The following theorem provides the transversality condition for the model \eqref{1.1}.
\begin{theorem}\label{TH7.3}
	Let Assumption \ref{ASS7.1} be satisfied and the parameters $\mu> 0, \tau> 0, \kappa> 0, n\in \mathbb{N}$ be fixed.
	For each $k \geq 0$, if $\pm \mathrm{i} \omega_k  $ with $\omega_k > 0$ is the purely imaginary root of the characteristic equation associated to $\alpha_k$ defined in Proposition  \ref{PRO7.2}, then  there exists $\rho_k > 0$ and a $C^1$-map $\hat{\lambda}_k: (\alpha_k - \rho_k, \alpha_k + \rho_k) \to \mathbb{C}$ such that
	\begin{equation*}
	\hat{\lambda }_k(\alpha_k) = \mathrm{i} \omega_k, \quad \Delta(\hat{\lambda}_k(\alpha), \alpha ) = 0,\quad \forall \alpha \in (\alpha_k - \rho_k, \alpha_k + \rho_k)
	\end{equation*}
	satsfying the transversality condition
	\begin{equation}\label{7.14}
	\mathrm{Re} \frac{\dd \hat{\lambda}_k(\alpha_k ) }{ \dd \alpha } > 0.
	\end{equation}
\end{theorem}
\begin{proof}
	 %Let Assumption \ref{ASS7.1} be satisfied. Then the characteristic function is  $\Delta(\lambda, \alpha) = 0$, where
	  By differentiating the characteristic function $\Delta(\lambda, \alpha) $ with $\lambda$ and $\alpha$ and noticing \eqref{7.3}, we have
%    \begin{equation*}
%       \frac{\partial \Delta(\lambda, \alpha)}{ \partial \lambda} =(1-\ln \alpha) \tau e^{-\lambda\tau} \left( 1 + \frac{\lambda}{\mu +\kappa} \right)^{-n-1} + (1 -\ln \alpha) e^{-\lambda \tau} (n+1) \left( 1 + \frac{\lambda}{\mu+\kappa} \right)^{-n-2} \frac{1}{\lambda + \mu + \kappa}
%    \end{equation*}
%    and
%    \begin{equation*}
%       \frac{\partial \Delta(\lambda, \lambda)}{\partial \alpha} = \frac{1}{\alpha} e^{-\lambda\tau} \left( 1 + \frac{\lambda}{\mu +\kappa} \right)^{-n-1}.
%    \end{equation*}
	\begin{equation*}
	\frac{\partial \Delta(\lambda, \alpha) }{\partial \lambda } = %( 1- \ln \alpha ) e^{-\lambda \tau } \left( 1 + \frac{\lambda }{\mu + \kappa } \right)^{-n - 1}
	\tau + \frac{n + 1}{\lambda + \mu + \kappa }  , \quad  \frac{\partial \Delta(\lambda, \alpha) }{\partial \alpha } =\frac{1 }{\alpha (1 - \ln \alpha) } .
	\end{equation*}
	Therefore, by setting $(\lambda, \alpha )= (\mathrm{i} \omega_k, \alpha_k)$ in  above equations, we have
	\begin{equation}\label{7.15}
	\mathrm{Re} \frac{\partial \Delta (\mathrm{i} \omega_k, \alpha_k )}{\partial \lambda} = \tau + \frac{  (n +1) ( \mu + \kappa)}{(\mu + \kappa)^2+ \omega_k^2} > 0, \quad \frac{\partial \Delta(\mathrm{i}\omega_k, \alpha_k) }{\partial \alpha } =\frac{1 }{\alpha_k (1 - \ln \alpha_k)} < 0.
	\end{equation}
	It follows from the implicit function theorem around each $(\mathrm{i} \omega_k, \alpha_k )$ that for each $k \geq 0$, there   exists $\rho_k > 0$ and a $C^1$-map $\hat{\lambda}_k: (\alpha_k - \rho_k, \alpha_k + \rho_k) \to \mathbb{C}$ such that
	\begin{equation*}
	\hat{\lambda }_k(\alpha_k) = \mathrm{i} \omega_k, \quad \Delta(\hat{\lambda}_k(\alpha), \alpha ) = 0,
	\quad \forall \alpha \in (\alpha_k - \rho_k, \alpha_k + \rho_k).
	\end{equation*}
	Furthermore,  by differentiating  $\Delta(\hat{\lambda}(\alpha), \alpha) = 0$ with $\alpha$, we can obtain
	\begin{equation}\label{7.16}
	\frac{\partial \Delta( \hat{\lambda}_k(\alpha ), \alpha )}{\partial \alpha } + \frac{\partial \Delta( \hat{\lambda}_k(\alpha ), \alpha )}{\partial \lambda } \frac{\dd \hat{\lambda}_k(\alpha )}{\dd \alpha }= 0,
	\quad \forall \alpha \in (\alpha_k - \rho_k, \alpha_k + \rho_k).
	\end{equation}
    The result \eqref{7.14} can be obtained by substituting \eqref{7.15} into \eqref{7.16} with $\alpha = \alpha_k$.
\end{proof}

By Proposition \ref{PRO7.2} and Theorem \ref{TH7.3} and using the Hopf bifurcation theorem, we can immediately  obtain the following Hopf bifurcation result.
\begin{theorem}
	Let Assumption \ref{ASS7.1} be satisfied. Then, for each $k\geq 0$, the number $\alpha_k$ defined in Proposition \ref{PRO7.2} is a Hopf bifurcation point for the system \eqref{1.1} parameterized by $\alpha$  around the positive equilibrium $\bar{u}$.
\end{theorem}

\end{document}